\documentclass[11pt, reqno]{amsart}
\usepackage[utf8]{inputenc}
\usepackage{amsmath}
\usepackage{amsthm}
\usepackage{amsfonts}
\usepackage{amssymb}
\usepackage{mathrsfs, mathtools, mathdots}
\usepackage{diagbox}

\usepackage[top=30mm,bottom=30mm,left=30mm,right=30mm]{geometry}
\newtheorem{theorem}{Theorem}[section]

\newtheorem{corollary}{Corollary}[section]

\newtheorem{definition}{Definition}

\newtheorem*{theorem*}{Theorem}
\newtheorem*{remark*}{Remark}
\newtheorem*{problem*}{Problem}
\newtheorem*{conjecture*}{Conjecture}
\newtheorem{lemma}[theorem]{Lemma}

\newcommand{\SSS}{\mathcal{S}}

\renewcommand{\P}{{\mathcal P}}
\renewcommand{\mod}{{\rm mod\, }}

\title[An effective bound on Generalized Diophantine $m$-tuples]{An effective bound on Generalized Diophantine $m$-tuples}
\author{Saunak Bhattacharjee}
\author{Anup B. Dixit}
\author{Dishant Saikia}

\address{IISER Tirupati, C/O Sree Rama Engineering College (Transit Campus), Tirupati, Andhra Pradesh, India 517507.}
\email{saunakbhattacharjee@students.iisertirupati.ac.in}
\address{Institute of Mathematical Sciences, CIT Campus, Taramani, Chennai, Tamil Nadu, India 600113.}
\email{anupdixit@imsc.res.in}
\address{Freie Universität Berlin,  Kaiserswerther Str. 16-18, Berlin, Germany 14195.}
\email{saikiadishant@gmail.com}
\thanks{The research of the second author is partially supported by the Inspire Faculty Fellowship. The research of the first and third authors were supported by a summer research program in IMSc, Chennai.}
\begin{document}

\date{\today}
\maketitle

\begin{abstract}
For non-zero integers $n$ and $k\geq2$, a generalized Diophantine $m$-tuple with property $D_k(n)$ is a set of $m$ positive integers $S = \{a_1,a_2,\ldots, a_m\}$ such that $a_ia_j + n$ is a $k$-th power for $1\leq i< j\leq m$. 
Define $M_k(n):= \sup\{|S| : S$ has property $D_k(n)\}$. In a recent work, the second author, S. Kim and M. R. Murty proved that $M_k(n)$ is $O(\log n)$, for a fixed $k$, as we vary $n$. In this paper, we obtain effective upper bounds on $M_k(n)$. In particular, we show that for $k\geq 2$, $M_k(n) \leq 3\,\phi(k)\, \log n$, if $n$ is sufficiently larger than $k$.
\end{abstract}
\bigskip

\section{\bf Introduction}
\medskip

Given a non-zero integer $n$, we say a set of natural numbers $S=\{a_1,a_2,\ldots, a_m\}$ is a Diophantine $m$-tuple with property $D(n)$ if $a_ia_j +n$ is a perfect square for all $1\leq i<j\leq m$. Diophantus first studied such sets of numbers and found the quadruple $\{1,33,68,105\}$ with property $D(256)$. The first $D(1)$-quadruple $\{1,3,8,120\}$ was discovered by Fermat, and this was later generalized by Euler, who found the following infinite family of quadruples with property $D(1)$, namely
\begin{equation*}
    \{a, b, a + b + 2r, 4r(r + a)(r + b) \},
\end{equation*}
where $ab + 1 = r^2$. In fact, it is known that any $D(1)$-triple can be extended to a Diophantine quadruple \cite{AHS}. In 1969, using Baker's theory of linear forms in logarithm of algebraic numbers and a reduction method based on continued fractions, Baker and Davenport \cite{Baker-Davenport} proved that Fermat's example is the only extension of $\{1,3,8\}$ with property $D(1)$. In 2004, Dujella \cite{Dujella2004}, using similar methods, proved that there are no $D(1)$-sextuples and there are only finitely many $D(1)$-quintuples, if any. The conjecture on the non-existence of $D(1)$-quintuples was finally settled in $2019$ by He, Togb\'e, and Ziegler in \cite{HTZ}.\\

However, in general, there are $D(n)$-quintuples for $n\neq 1$. For instance, $$\{1, 33, 105, 320, 18240\}\quad \text{ and} \quad  \{5, 21, 64, 285, 6720\}$$  are Diophantine quintuples satisfying property $D(256)$. Also, we note that there are known examples of $D(n)$-sextuples, but no $D(n)$-septuple is known. So, it is natural to study the size of the largest tuple with property $D(n)$. Define
\begin{equation*}
    M_n := \sup\{|S| : S \text{ satisfies property } D(n)\}.
\end{equation*}
In 2004, Dujella \cite{dujella} showed that 
\begin{equation*}
    M_n \leq C \log |n|,
\end{equation*}
where $C$ is an absolute constant. He also showed that for $n>10^{100}$, one can choose $C=8.37$. This constant was improved by Becker and Murty \cite{Becker-Murty}, who showed that for any $n$,
\begin{equation}\label{becker-ram}
    M_n \leq 2.6071 \log |n| + \mathcal{O}\left(\frac{\log |n|}{(\log\log |n|)^2}\right).
\end{equation}
\par
Our goal is to study this problem when squares are replaced by higher powers.
\begin{definition}[generalized Diophantine $m$-tuples]  Fix a natural number $k\geq 2$.  
A set of natural numbers $S=\{a_1,a_2,\ldots, a_m\}$ is said to satisfy property $D_k(n)$ if $a_ia_j +n$ is a $k$-th power for all $1\leq i<j\leq m$.
\end{definition}
We analogously define the following quantity for each non-zero integer $n$,
\begin{equation*}
    M_k(n):= \sup\{|S| : S \text{ satisfies property }D_k(n)\}.
\end{equation*}
For $k\geq 3$ and $m\geq 3$, we can apply the celebrated theorem of Faltings \cite{faltings} to deduce that a superelliptic curve of the form
\begin{equation*}
    y^k = f(x)=(a_1 x +n)(a_2 x +n) (a_3 x+n)(a_4 x + n)~\cdots~ (a_m x +n)
\end{equation*}
has only finitely many rational points and a fortiori, finitely many integral points. Therefore, a set $S$ satisfying property $D_k(n)$ must be finite. When $n=1$, Bugeaud and Dujella \cite{Bugeaud-Dujella} showed that
\begin{equation*}
    M_{3}(1) \leq 7, \quad M_{4}(1) \leq 5,  \quad M_{k}(1) \leq 4 ~  \text{for}~ 5 \leq k \leq 176,~ \text{and}~ M_{k}(1) \leq 3~ \text{for}~ k \geq 177.
\end{equation*}
In other words, the size of $D_k(1)$-tuples is bounded by $3$ for large enough $k$. In the general case, for any $n \neq 0$ and $k \geq 3$, B\'{e}rczes, Dujella, Hajdu and Luca \cite{dujella-luca} obtained upper bounds for $M_k(n)$. In particular, they showed that for $k\geq 5$
\begin{equation*}
    M_k(n) \leq 2|n|^5 + 3.
\end{equation*}
\medskip

In \cite{dixit-kim-murty}, the second author, S. Kim and M. R. Murty improved the above bounds on $M_k(n)$ for large $n$ and a fixed $k$. Define
\begin{equation*}
    M_k(n;L):=\sup\{|S \cap [n^L,\infty)|: S \text{ satisfies property } D_k(n)\}
\end{equation*}
They proved that for $k\geq 3$, as $n\to\infty$,
\begin{equation}\label{dixit-kim-murty}
    M_k(n,L)\ll_{k,L} 1, \text{ for }L\geq 3 \,\,\,\, \text{ and }\,\,\,\, M_k(n) \ll_k \log n.
\end{equation}
The purpose of this paper is to explicitly obtain the implied constants in \eqref{dixit-kim-murty}. In \cite{dixit-kim-murty}, the bounds for $M_k(n)$ were proven under the further assumption that $n>0$. However, this is not necessary, as was remarked in \cite{dixit-kim-murty}, but an argument was not provided. We first prove the bounds \eqref{dixit-kim-murty} for all non-zero integers $n$.

\begin{theorem}\label{thm-2}
    Let $k \geq 3$ be an integer. Then the following holds as $|n| \to \infty$:
    \begin{enumerate}
        \item For $L \geq 3$,
        $$M_k(n, L) \ll 1,$$
        where the implied constant depends on $k$ and $L$, but is independent of $n$.
        \medskip
        
        \item Moreover,
        $$M_k(n) \ll \log |n|$$
        where the implied constant depends on $k$.
    \end{enumerate}
\end{theorem}

We now state our main theorem, which is the effective version of Theorem \ref{thm-2}.

\begin{theorem}\label{thm-1}
    Let $k$ be a positive integer $\geq 3$. Then the following holds.
    \begin{enumerate}
        \item[(a)] For $L \geq 3$,
        \begin{equation}\label{abs-large}
        M_k(n, L) \leq 2^{28} \log (2k) \log (2 \log (2k)) + 14.
        \end{equation}
    
       \medskip

        \item[(b)] Suppose $n$ and $k$ vary such that as $|n|\to \infty$ and
        $$
            k = o(\log\log |n|).
        $$
        Then
        $$
        M_k(n) \leq 3 \, \phi(k)\, \log |n| + \mathcal{O} \left(\frac{ (\phi(k))^2\log |n|}{\log \log |n|}\right),
        $$
        where $\phi(n)$ denotes the Euler totient function.
    \end{enumerate}
\end{theorem}

{\bf Remarks.}
\begin{enumerate}
\item It is possible to replace $14$ on the right hand side of \eqref{abs-large} with a smaller positive integer for large values of $k$.
\item For a fixed $k>2$, from Theorem \ref{thm-1}(b), we obtain that as $|n|\to \infty$
$$
    M_k(n) \leq 3 \, \phi(k) \, \log |n| + \mathcal{O}\left(\frac{\log |n|}{\log\log |n|}\right).
$$
This upper bound is very close to the case $k=2$, where the best known upper bound due to Becker and Ram Murty is given by \eqref{becker-ram}.

\end{enumerate}
\bigskip

\section{\bf Preliminaries}
\medskip

In this section, we recall and develop the necessary tools to prove our main theorems.
%\medskip

\subsection{Gallagher's larger sieve}

In 1971, Gallagher \cite{Gallagher} discovered an elementary sieve inequality which he called the larger sieve.
We refer the reader to \cite{cm} for the general discussion but record the result in a form applicable to our context.

\begin{theorem}\label{larger}  Let $N$ be a natural number and 
 $\SSS$  a subset of $\{1,2,\ldots, N\}$.  Let ${\mathcal P}$ be a set of primes.
For each prime $p \in {\mathcal P}$, let $\SSS_p=\SSS \pmod{p}$.
For any $1<Q\leq N$, we have
\begin{equation}
 |\SSS|\leq \frac{\underset{p\leq Q, p\in \P}\sum\log p - \log N}{\underset{p\leq Q, p \in \P}\sum\frac{\log p}{|\SSS_p|}-\log N},
\end{equation}
where the summations are over primes $p\leq Q, p \in \P$ and the inequality holds
 provided the denominator is positive.
\end{theorem}
\medskip

\subsection{A quantitative Roth's theorem}
Quantitative results related to counting exceptions in Roth's celebrated theorem on Diophantine approximations were established by 
a variety of authors.  We will use the following result due to Evertse \cite{evertse}.
For an algebraic number $\xi$ of degree $r$, we define the (absolute) height by
$$H(\xi):= \left( a\prod_{i=1}^r \max(1, |\xi^{(i)}|)\right)^{1/r}, $$
where $\xi^{(i)}$ for $1\leq i\leq r$ are the conjugates (over $\mathbb{Q}$) and $a$ is the positive integer such that
$$a\prod_{i=1}^r (x-\xi^{(i)}) $$
has rational integer coefficients with gcd 1.  
\begin{theorem}\label{qr}
Let $\alpha$ be a real algebraic number of degree $r$ over $\mathbb{Q}$, and $0 <\kappa \leq 1$.
The number of rational numbers $p/q$ satisfying $\max (|p|,|q|) \geq \max (H(\alpha), 2)$,
\begin{equation*}
    \bigg|\alpha - \frac{p}{q} \bigg| \leq \frac{1}{\max (|p|,|q|)^{2+\kappa}}
\end{equation*}
is at most 
$$2^{25}\kappa^{-3}\log (2r) \log (\kappa^{-1} \log (2r)).$$

\end{theorem}

\medskip

\subsection{Vinogradov's theorem}
The following bound on character sums was proved by Vinogradov (see \cite{vinogradov}). 
\begin{lemma}\label{vinogradov}
Let $\chi \pmod q$ be a non-trivial Dirichlet character and $n$ be an integer such that $(n,q)=1$. If $\mathcal{A} \subseteq (\mathbb{Z}/q\mathbb{Z})^*$ and $\mathcal{B} \subseteq (\mathbb{Z}/q\mathbb{Z})^* \cup \{0\}$, then
\begin{equation*}
\sum_{a\in\mathcal{A}} \sum_{b\in\mathcal{B}} \chi(ab + n) \leq \sqrt{q |\mathcal{A}| |\mathcal{B}|}.
\end{equation*}
\end{lemma}
The original method of Vinogradov does not produce the bound above and instead gives the right hand side as $ \sqrt{2q |\mathcal{A}| |\mathcal{B}|}$. However, the above bound holds and a short proof of this can be found in \cite[Proposition 2.5]{Becker-Murty}.
\medskip

\subsection{Prime bounds in arithmetic progression}

Let $Q,k,a$ are positive integers with $(a,k)=1$. Denote by $\theta(Q;k,a)$ the sum of the logarithms of the primes $p \equiv a \pmod k$ with $p \leq Q$, i.e., 
$$
    \theta(Q;k,a) := \sum_{\substack{p \equiv a \bmod k \\ p \text{ prime} \leq Q}} \log p. 
$$
We will need the following bound on $\theta(Q;k,a)$ obtained by Bennet, Martin, O’Bryant, Rechnitzer in \cite[Theorem 1.2]{bmbr}:
\begin{theorem}\label{pnt}
    For $k \geq 3$ and $(a, k) = 1$, 
    \begin{equation}
        \left |\theta (Q; k, a) - \frac{Q}{\phi(k)} \right| < \frac{1}{160} \frac{Q}{\log Q}
    \end{equation}
    for all $Q \geq Q_0 (k)$ where 
    \begin{equation*}
        Q_0(k) = 
        \begin{cases*}
        8 \cdot 10^9 & if $3 \leq k \leq 10^5$ \\
        \exp(0.03\sqrt{k}\log^3 k) & if $k > 10^5$.
        \end{cases*}
    \end{equation*}
\end{theorem}

\subsection{Gap principle}
\medskip

The next two lemmas are variations of a gap principle of Gyarmati \cite{gyarmati}. The following lemma was proved in \cite{dixit-kim-murty}.
\begin{lemma}[\cite{dixit-kim-murty}, Lemma 2.4]\label{gyar} Let $k\geq 2$.  Suppose that $a,b,c,d$ are positive integers such that $a<b$ and $c<d$.  Suppose further that
$$ac+n, \quad bc+n, \quad  ad+n, \quad bd+n $$
are perfect $k$-th powers.  Then, 
$$bd \geq k^k n^{-k} (ac)^{k-1}. $$
\end{lemma}
An immediate Corollary of this lemma shows that ``large'' elements of any set with property $D_k(n)$ have ``super-exponential growth.''
\begin{corollary}[\cite{dixit-kim-murty} Corollary 2]\label{super}  Let $k\geq 3$ and $m\geq 5$.   Suppose that $n^3 \leq a_1 < a_2 < \ldots < a_m$ and the set $\{ a_1, a_2, ..., a_m\}$ has property $D_k(n)$.  Then $a_{2+3j} \geq a_2^{(k-1)^{j}}$ provided $1\leq j \leq (m-2)/3$.
\end{corollary}

A modification of the proof of the above Lemma \ref{gyar} yields a gap principle for negative values of $n$.
\begin{lemma}\label{abcd-lemma}
    For $n > 0$ and natural numbers $a,b,c,d$ such that $n^3 \leq a < b < c < d$, we have
    \begin{equation*}
        (ac-n)(bd-n)  \geq \frac{abcd}{2}.
    \end{equation*}
\end{lemma}

\begin{proof}
    Since $$
            (ac-n)(bd-n) = abcd-n(ac+bd)+n^2,
            $$ 
        it is enough to prove that 
        $$
            \frac{abcd}{2} \geq n(ac+bd) - n^2.
        $$
    Also, since $a \geq n^3$ and $c > n^3$, we get for all cases other than $n=1, a=1, b=2, c=3$,
    \begin{align*}
        abcd &\geq 4n bd\\
        &\geq 2nbd +2nac\\
        &\geq 2nbd+2nac-2n^2
    \end{align*}
    where the first inequality is obvious as $a \geq n^3$ and $c \geq n^3+2$. Thus, we get the desired result.
    For the case $n=1, a=1, b=2, c=3$, clearly
    \begin{equation*}
        2n(ac+bd) - 2n^2 = 4+4d < 6d = abcd
    \end{equation*}
    as $d > c$.
\end{proof}

We are now ready to prove the following analog of Lemma \ref{gyar}.
\begin{lemma}\label{gap}
Let $n > 0$ and $k \geq 2$. Suppose that $a,b,c,d$ are positive integers such that $n^3 \leq a < b < c < d$. Suppose further that $ac - n, bc-n, ad-n, bd-n$ are perfect $k$-th powers. Then, 
$$bd \geq k^k 2^{-k} n^{-k} (ac)^{k-1}.$$
\end{lemma}
\begin{proof}
Since $(b-a)(d-c)>0$, we have
$$
    bd + ac > ad + bc,
$$
and it is easily seen that 
$$
    (ad-n)(bc-n) > (ac-n)(bd-n).
$$
As $(ac-n)(bd-n)$ and $(ad-n)(bc-n)$ are both perfect $k$-th powers, we have
\begin{align*}
    (ad-n)(bc-n) &\geq [((ac-n)(bd-n))^{1/k} + 1]^k\\
    &\geq (ac-n)(bd-n) + k((ac-n)(bd-n))^{k-1/k}\\
    &\geq (ac-n)(bd-n) + k\left(\frac{abcd}{2}\right)^{k-1/k}
\end{align*}
where the last inequality follows from Lemma \ref{abcd-lemma}. Thus,
\begin{equation*}
    -n(ad+bc) \geq -n(ac+bd) + k\left(\frac{abcd}{2}\right)^{k-1/k}.
\end{equation*}
Since $bd > ad + bc - ac > 0$, we have $bd+ac-ad-bc < bd$ and hence we obtain
$$
    nbd > k\left(\frac{abcd}{2}\right)^{k-1/k}.
$$ 
Therefore,
$$
    bd \geq k^k 2^{1-k} n^{-k} (ac)^{k-1} \geq k^k 2^{-k} n^{-k} (ac)^{k-1},
$$
which proves the lemma.
\end{proof}
This enables us to prove super-exponential growth for large elements of a set with $D_k(n)$, when $n<0$.

\begin{corollary}
Let $k \geq 3$. If $n^3 \leq a < b < c < d < e$ are natural numbers such that the set $\{a,b,c,d,e\}$ has property $D_k(-n)$, then $e \geq b^{k-1}$.
\end{corollary}

\begin{proof}
From Lemma \ref{gap},
\begin{equation*}
    ce \geq k^k 2^{-k} n^{-k} (bd)^{k-1} \geq k^k 2^{-k} n^{-k} (bc)^{k-1}. 
\end{equation*}
Therefore, $$e \geq b^{k-1} c^{k-2} n^{-k} 2^{-k} \geq b^{k-1} n^{2k-6} 2^{-k} \geq b^{k-1}.$$
\end{proof}
Using induction on the previous corollary, we deduce
\begin{corollary}\label{gap1}
    Let $k \geq 3$ and $m \geq 5$. Suppose that $n^3 \leq a_1 < a_2 < \ldots < a_m$ and the set $\{a_1, a_2, \ldots, a_m\}$ has property $D_k(-n)$. Then $a_{2+3j} \geq a_2^{(k-1)^j}$ provided $1 \leq j \leq (m-2)/3$.
\end{corollary}
\bigskip

\section{\bf Proof of the main theorems}
\subsection{Proof of Theorem \ref{thm-2}}

We first prove Theorem \ref{thm-2} as the proof follows a similar method as in \cite{dixit-kim-murty}.\\

Let $n$ be a positive integer and $m = M_k (-n)$ and $S = \{a_1, a_2, a_3, \cdots, a_m\}$ be a generalized $m$-tuple with the property $D_k(-n)$. Suppose $n^L < a_1 < a_2 < \cdots < a_m$ for some $L \geq 3$. Consider the system of equations

\begin{align}\label{eqns}
    &a_1 x -n = u^k\nonumber\\
    &a_2 x -n = v^k.
\end{align}
Clearly, $x=a_i$ for $i \geq 3$ are solutions to this system. Also, 
\begin{equation*}
    \left|a_2 u^k - a_1 v^k \right| = n(a_2-a_1).
\end{equation*}
Let $\alpha := (a_1/a_2)^{1/k}$ and $\zeta_k \coloneqq e^{2 \pi i/k}$. Then, we have the following two lemmas analogous to the ones proved in \cite{dixit-kim-murty}.

\begin{lemma}\label{lem1}
    Let $k \geq 3$ be odd. Suppose $u, v$ satisfy the system of equations \eqref{eqns}. Let 
    \begin{equation*}
        c(k) \coloneqq \prod_{j=1}^{(k-1)/2} \left(\sin \frac{2 \pi j}{k}\right)^2.
    \end{equation*}
    Then, for $n > 2^{1/(L-1)}c(k)^{-1/(L-1)}$,
    \begin{equation*}
        \left |\frac{u}{v} - \alpha\right| \leq \frac{a_2}{2v^k}.
    \end{equation*}
\end{lemma}
We omit the proof of this lemma here as it is identical to the proof of \cite[Lemma 3.1]{dixit-kim-murty}.
\begin{lemma}\label{lem2}
    Let $(u_i, v_i)$ denote distinct pairs that satisfy the system of equations \eqref{eqns} with $v_{i+1} > v_i$. For $n > 2^{1/(L-1)}c(k)^{-1/(L-1)}$ and $i\geq 14$,
    \begin{equation*}
        \left|\frac{u}{v} - \alpha\right| < \frac{1}{v_i^{k-1/2}},
    \end{equation*}
    and $v_i > a_2^4$.
\end{lemma}

\begin{proof}
    From Lemma \ref{lem1}, we have
    $\left|\frac{u_i}{v_i} - \alpha \right| < \frac{a_2}{2v_i^k}$. Thus, we need to show $a_2 < 2v_i^{1/2}$ for $i > 14$. Since $v_i^k = a_2 a_i - n$, we have $v_i \geq a_i^{1/k}$. By Corollary \ref{gap1},
    $$
        a_{2+3j} \geq a_2^{(k-1)^j}
    $$ 
    so that $v_{2+3j} \geq a_2^{(k-1)^j}$. We choose a positive integer $j_0$ such that $(k-1)^{j_0} > 4k$. Since $k \geq 3$, $j_0=4$ satisfies the condition. As $2+3j_0 = 14$, we have $v_i \geq v_{14} > a_2^4$ for all $i \geq 14$. This completes the proof. 
\end{proof}
For larger values of $k$, the number $14$ in the above Lemma can be improved to $2 + 3 j_0$, where $j_0$ satisfies the condition $(k-1)^{j_0} > 4k$.\\

Now, assume that $(u_1,v_1), (u_2, v_2), \cdots, (u_m,v_m)$ satisfy the system of equations \eqref{eqns} with 
$$
    v_i > \max(a_2^{1/k}, 2) \geq \max (H(\alpha), 2). 
$$
By Lemma \ref{lem2}, for $14 \leq i \leq m$,
\begin{equation*}
    \left|\frac{u}{v} - \alpha\right| < \frac{1}{v_i^{k-1/2}} \leq \frac{1}{v_i^{2.5}},
\end{equation*}
as $k \geq 3$. Since $\alpha=(a_1/a_2)^{1/k} < 1$ and $\max (u_i, v_i) = v_i$, from Theorem \ref{qr}, the number of such $i$'s is $\mathcal{O} (\log k \,\log \log k)$. This proves Theorem \ref{thm-2}.
\medskip

\subsection{Proof of Theorem \ref{thm-1}}

Let $m = M_k (n)$ and $S = \{a_1, a_2, a_3, \ldots a_m\}$ be a generalized $m$-tuple with the property $D_k(n)$. Suppose $|n|^L < a_1 < a_2 < \ldots < a_m$ for some $L \geq 3$. We consider the system of equations
\begin{align}\label{eqns2}
    & a_1 x +n = u^k \nonumber\\
    & a_2 x +n = v^k.
\end{align}
    
As before, $x=a_i$ for $i \geq 3$ are solutions to this system. The statements of Lemma \ref{lem1} and \ref{lem2} hold for all non-zero integers $n$. For $n>0$ this was proved in \cite{dixit-kim-murty}. 
\medskip

\begin{proof}[Proof of Theorem \ref{thm-1}(a)]

Let $(u_1,v_1), \cdots (u_m,v_m)$ satisfy the system of equations \eqref{eqns2} with $v_i > \max (a_2^{1/k}, 2) \geq \max (H(\alpha), 2)$. By Lemma \ref{lem2}, we get for $14 \leq i \leq m$,
\begin{equation*}
    \left |\frac{u_i}{v_i} - \alpha \right| \leq \frac{1}{v_i^{k-1/2}} \leq \frac{1}{v_i^{2.5}},
\end{equation*}
as $k \geq 3$. Since $\alpha = (a_1/a_2)^{1/k} < 1$ and $\max (u_i, v_i) = v_i$, applying Theorem \ref{qr} with $\kappa = 0.5$, we get that the number of $i$'s satisfying the above inequality is 
$$2^{25} (0.5)^{-3} \log (2k) \log ((0.5)^{-1} \log (2k))$$
which is
$$
    2^{28} \log (2k) \log (2 \log (2k)).
$$
So, the  total number of solutions is at most 
$$
    2^{28} \log (2k) \log (2 \log (2k)) + 14
$$ 
for $k \geq 3$.
\end{proof}
\medskip

\begin{proof}[Proof of Theorem \ref{thm-1} (b)]
Let $S = \{a_1, a_2, \ldots, a_m\}$ be a generalized Diophantine $m$-tuple with property $D_k(n)$ such that each $a_i \leq |n|^3$. We shall apply Gallagher's larger sieve with primes $p \leq Q$ satisfying $p \equiv 1 \bmod k$. Let $\mathcal{P}$ be the set of all primes $p \equiv 1\bmod k$. For all such primes $p \in \mathcal{P}$, there exists a Dirichlet character $\chi (\mod p)$ of order $k$. \\ 

Denote by $S_p$ the image of $S \pmod p$ for a given prime $p$. For $p\in \mathcal{P}$, applying Lemma \ref{vinogradov} with $\mathcal{A} = \mathcal{B} = S_p$ and $\chi \bmod p$ a character of order $k$, we obtain
\begin{equation*}
    |S_p|\left(|S_p| - 1\right) \leq \sum_{a \in S_p-\{0\}} \sum_{b \in S_p} \chi(ab+n) + |S_p| \leq \sqrt{p} |S_p| + |S_p|. 
\end{equation*}
Thus, 
$$
|S_p| \leq \sqrt{p}+2.
$$
Take $N = |n|^3$. Since $a_i \leq |n|^3$, applying Theorem \ref{larger}, we obtain
\begin{equation*}
    |S| \leq \frac{\sum\limits_{p\in \mathcal{P}, \,\, p \leq Q } \log p- \log N}{\sum\limits_{p\in \mathcal{P}, \, \, p \leq Q } \frac{\log p}{|S_p|}-\log N}.
\end{equation*}
By Theorem \ref{pnt}, 
    $$
        \sum_{\substack{p \leq Q \\ p \equiv 1 \mod k}} \log p = \frac{Q}{\phi(k)} + \mathcal{O} \left(\frac{Q}{\log Q}\right)
    $$ 
when $Q > Q_0(k)$. Using partial summation,
    $$
        \sum_{\substack{p \leq Q \\ p \equiv 1 \mod k}} \frac{\log p}{\sqrt{p}+2} = \frac{2 \sqrt{Q}}{\phi(k)} + \mathcal{O} \left(\frac{\sqrt{Q}}{\log Q}\right).
    $$
Thus,
\begin{equation}\label{eq-22}
    |S| \leq \frac{\frac{Q}{\phi(k)} + \mathcal{O}\left(\frac{Q}{\log Q}\right) - \log N}{\frac{2 \sqrt{Q}}{\phi(k)} + \mathcal{O}\left(\frac{\sqrt{Q}}{\log Q}\right) - \log N}.
\end{equation}
Choose $Q=(\phi(k) \log N)^2$. Note that the condition $Q> Q_0(k)$ is same as
\begin{equation}\label{eq-12}
     \log N  > \frac{\exp\left(0.015 \sqrt{k} (\log k)^3\right) }{ \phi(k)}.
\end{equation}
Since $k = o(\log\log |n|)$, \eqref{eq-12} holds for $N$ large enough. Now, for both the numerator and the denominator, multiply by $\phi(k)$ and divide by $\log N$, to get 
\begin{equation}\label{eq-10}
    |S| \leq \frac{\phi(k)\log N -1 + \mathcal{O} \left(\frac{Q}{\log N \log Q}\right)}{1 + \mathcal{O}\left(\frac{\sqrt{Q}}{\log N \log Q}\right)}.
\end{equation}
Because $k=o(\log\log N)$, it is easy to see that
$$
    \frac{\sqrt{Q}}{\log N \log Q} = o(1).
$$
Hence, the denominator in \eqref{eq-10} is $1 + o(1)$ and we obtain
$$
 |S| \leq \phi(k) \log N + \mathcal{O} \left(\frac{ (\phi(k))^2\log N}{\log \log N}\right).
$$
Since $N=|n|^3$,
\begin{equation*}
    M_k(n) \leq 3 \, \phi(k) \, \log |n| + \mathcal{O} \left(\frac{(\phi(k))^2 \log |n|}{\log \log |n|}\right)
\end{equation*}
as required.
\end{proof}

\section*{\bf Acknowledgements}
We thank Prof. Alain Togb\'{e} for his question regarding effective versions of the bounds in \cite{dixit-kim-murty} during the second author's talk in the Leuca 2022 conference, which motivated this paper.

% \bibliographystyle{abbrv}

% \bibliography{biblio}

\end{document}